\documentclass{article}
\usepackage[utf8]{inputenc}     % for éô
\usepackage[english]{babel}     % for proper word breaking at line ends
\usepackage[a4paper, left=1.5in, right=1.5in, top=1.5in, bottom=1.5in]{geometry}

\usepackage{titlesec}
\titleformat{\paragraph}
{\normalfont\normalsize\bfseries}{\theparagraph}{1em}{}
\titlespacing*{\paragraph}
{0pt}{3.25ex plus 1ex minus .2ex}{1.5ex plus .2ex}

\usepackage{tikz}
\usetikzlibrary{arrows,positioning,shapes.geometric}
\usepackage[useregional]{datetime2}             
             
\usepackage{graphicx}           % for ?
\usepackage{amsmath,amssymb}    % for better equations
\usepackage{amsthm}             % for better theorem styles
\usepackage{amsmath}
\usepackage{mathtools}          % for greek math symbol formatting
\usepackage{enumitem}           % for control of 'enumerate' numbering
\usepackage{listings}           % for control of 'itemize' spacing
\usepackage{todonotes}          % for clear TODO notes
\usepackage{hyperref}  
\usepackage{bbold}% page numbers and '\ref's become clickable

\newtheorem{theorem}{Theorem}[section]

\newtheorem{prop}[theorem]{Proposition}
\newtheorem{lemma}[theorem]{Lemma}

\title{Almost sure upper bound for a model problem for multiplicative chaos in number theory}
\author{Rachid Caich }
\date{\today}

\begin{document}

\maketitle
\begin{abstract}
The goal of this work is to prove a new  sure upper bound in a setting that can be
thought of as a simplified function field analogue. This result is comparable to a recent result of the author concerning  almost sure upper bound of random multiplicative functions. Having a simpler quantity allows us to make the proof more accessible.
\end{abstract}

\section{Introduction}
Let $(X(k))_{k\geqslant 1}$ be a sequence of independent standard complex Gaussian random variables, where the real and imaginary parts of $X(k)$ are independently distributed like real Gaussian random variables with mean $0$ and variance $\frac{1}{2}$. Consider a sequence of random variables $(A(n))_{n\geqslant 0}$ defined by the formal power series identity 
\begin{equation}\label{link_xk_an}
    \exp\bigg( \sum_{k=1}^{+\infty} \frac{X(k)}{\sqrt{k}}z^k \bigg)= \sum_{n=0}^{+\infty} A(n)z^n.
\end{equation}
Let $\mathcal{P}$ be the set of the prime numbers, \textit{a Steinhaus random multiplicative function} is obtained by letting $(f(p))_{p\in \mathcal{P}}$ be a sequence of independent Steinhaus random variables (i.e distributed uniformly on the unit circle $\{ |z|=1\}$), and then setting
$$ f(n):= \prod_{p^{a}||n} f(p)^{a}  \text{ for all } n\in \mathbb{N},$$
where $ p^{a}||n$ means that $p^{a}$ is the highest power of $p$ that divides $n$. In two recent papers (\cite{model_zaman} and \cite{model_maxim}), the sequence of random variables $(A(n))_{n\geqslant 0}$ has been interpreted as an analogue to the Steinhaus random multiplicative function. \\
Recently, there has been much focus regarding the moments and almost sure bounds for the mean values of random multiplicative functions. For the lower bound, Harper \cite{Harper3} proved, using a \textit{Multiplicative Chaos} techniques, that for any function $V(x)$ tending to infinity with $x$, there almost surely exists arbitrarily large values of $x$ for which
\begin{equation}\label{Harper_lower_b}
    \big|M_f(x) \big| \gg \frac{\sqrt{x} (\log_2 x)^{1/4}}{V(x)}.
\end{equation}
where $M_f(x):=\sum_{n\leqslant x}f(n)$. Here and in the sequel $\log_k$ denotes the k-fold iterated logarithm. In \cite{Harper}, Harper proved when $x \to +\infty$
\begin{equation}\label{Harper_approximation}
    \mathbb{E}\bigg[ \big| M_f(x)\big| \bigg] \asymp \frac{\sqrt{x}}{(\log_2 x)^{1/4}}.
\end{equation}
This discrepancy of a factor $\sqrt{\log_2 x}$ between the first moment and the almost sure behaviour is similar to the Law Iterated Logarithm for independent random variables. For this reason Harper conjectured that for any fixed $\varepsilon >0$, we might have almost surely, as $x \to +\infty$
\begin{equation}\label{Harper_upper_b}
    M_f(x)\ll \sqrt{x} (\log_2 x)^{1/4+\varepsilon}
\end{equation}
(see the introduction in \cite{Harper3} for more details). The author in \cite{Rachid} proved 
\begin{equation}\label{Rachid_upper_bound}
    M_f(x)\ll \sqrt{x} (\log_2 x)^{3/4+\varepsilon}.
\end{equation}
In \cite{model_zaman}, Soundararajan and Zaman were motivated by the outcome and examined the moments of $A(n)$, revealing that they resemble those in the random multiplicative functions. They proved the analogue of \eqref{Harper_approximation}, that we have
$$
\mathbb{E}[|A(n)|] \asymp \frac{1}{(\log n)^{1/4}}.
$$
More recently, in \cite{model_maxim}, Gerspach, proved the analogue of \eqref{Harper_lower_b} that for any function $V(n)$ tending to infinity with $n$, there almost surely exist arbitrarily large values of $n$ for which
$$
|A(n)| \geqslant \frac{(\log n)^{1/4}}{V(n)}.
$$
The main goal of this paper is to prove the analogue of the almost sure inequality \eqref{Rachid_upper_bound}.
\begin{theorem}\label{theoreme_principal}
Let $\varepsilon >0$. Let $(A(n))_{n\geqslant 0}$ as defined in \eqref{link_xk_an}. We have almost surely,  as $n$ tends infinity
\begin{equation}
    A(n) \ll (\log n)^{\frac{3}{4}+ \varepsilon}.
\end{equation}
\end{theorem}
\noindent We aim to enhance and simplify the demonstration of theorem 1.1 given in \cite{Rachid}, in the case for the Steinhaus and Rademacher multiplicative function. Our objective is to create a gentle  proof for those seeking a comprehensible grasp of the theorem in the context of this model.
\section{Preliminaries}
\subsection{Notation}
Let's start by some definitions. Let $(\Omega, \mathcal{F}, \mathbb{P})$ be a probabilistic space. We call a \textit{filtration} every sequence $(\mathcal{F}_n)_{n \geqslant 1}$ of increasing sub-$\sigma$-algebras of $\mathcal{F}$. We say that a sequence of real random variables $(Z_n)_{n \geqslant 1}$ is \textit{submartingale} (resp. \textit{supermartingale}) sequence with respect to the filtration $(\mathcal{F}_n)_{n \geqslant 1}$, if the following properties are satisfied:\\
- $Z_n$ is $\mathcal{F}_n$ measurable\\
- $\mathbb{E}[|Z_n|]< +\infty$\\
- $\mathbb{E}[Z_{n+1} \, | \, \mathcal{F}_n] \geqslant Z_n$ (resp. $\mathbb{E}[Z_{n+1}\, | \,\mathcal{F}_n] \leqslant  Z_n$ ) almost surely.\\
We say that $(Z_n)_{n \geqslant 1}$ is martingale difference sequence with respect to the same filtration $(\mathcal{F}_n)_{n \geqslant 1}$ if \\
- $Z_n$ is $\mathcal{F}_n$ measurable\\
- $\mathbb{E}[|Z_n|]< +\infty$\\
- $\mathbb{E}[Z_{n+1} \, | \,\mathcal{F}_n] =0$ almost surely.\\
An event $ E \in \mathcal{F}$ happens \textit{almost surely} if  $\mathbb{P}[E]=1$. \\
Let $Z$ be a random variable and let $\mathcal{H}_1\subset \mathcal{H}_2 \subset \mathcal{F}$ some sub-$\sigma$-algebras,  we have
$$
\mathbb{E}\big[ \mathbb{E}\big[ Z \, \big| \,\mathcal{H}_2 \big] \,\big| \, \mathcal{H}_1  \big]=\mathbb{E}\big[ Z\, \big| \, \mathcal{H}_1 \big].
$$
\subsection{Some properties}
We follow the notations of Soundararajan and Zaman in \cite{model_zaman}. Let $(X(k))_{k\geqslant 1}$ be a sequence of independent standard complex Gaussian random variables. By a partition $\lambda $ we mean a non-increasing sequence of non-negative integers $\lambda_1 \geqslant \lambda_2 \geqslant ...,$ with $\lambda_n =0$ from a certain point onwards. We denote by $|\lambda| = \lambda_1 + \lambda_2 + \lambda_3 ...,$ and for each integer $k \geqslant 1$ we denote by $m_k=m_k(\lambda)$ the number of parts in $\lambda$ that equal to $k$. 
With this notations, let
\begin{equation}\label{a(lambda)_def}
    a(\lambda):= \prod_{k\geqslant 0} \bigg( \frac{X(k)}{\sqrt{k}}\bigg)^{m_k} \frac{1}{m_k! },
\end{equation}
we have then
$$
\exp\bigg( \sum_{k=1}^{+\infty} \frac{X(k)}{\sqrt{k}}z^k \bigg)= \sum_{\lambda} a(\lambda) z^{|\lambda|},
$$
thus, for every $n\geqslant 0$
$$
A(n)=\sum_{|\lambda|=n} a(\lambda).
$$
Note that for a standard complex Gaussian $Z$, we have
$$
\mathbb{E}\big[Z^n \overline{Z}^m\big]=\left\{
  \begin{array}{@{}ll@{}}
    n!  & \text{if}\ n =m, \\
    0 & \text{otherwise.}
  \end{array}\right.
$$
It follows that if $\lambda$ and $\lambda^{\prime}$ for two different partitions 
$$
\mathbb{E}\big[a(\lambda)\overline{a(\lambda^{\prime})} \big] =0.
$$
If $\lambda=\lambda^{\prime}$ then
$$\mathbb{E}\big[ |a(\lambda)|^2\big]=\prod_{k\geqslant 0}  \frac{1}{k^{m_k}m_k!^2} \mathbb{E}\big[|X(k)|^{2m_k}\big]= \prod_{k\geqslant 0}  \frac{1}{k^{m_k}m_k!}.$$
Thus
$$
\mathbb{E}\big[ |A(n)|^2\big] = \sum_{|\lambda|=n} \mathbb{E}\big[|a(\lambda)|^2\big] =\sum_{|\lambda|=n} \prod_{k\geqslant 0}  \frac{1}{k^{m_k}m_k!} =1.
$$
The final step is derived from the well-known formula for calculating the number of permutations in the symmetric group $S_n$ whose cycle decomposition corresponds to the partition $\lambda$.\\
One can study $A(n)$ through the generating function. Note that by Cauchy's Theorem, for $n\leqslant R$, we have
$$
A(n)= \frac{1}{2 \pi i}\int_{|z|=1} F_R(z)\frac{{\rm d}z}{z^{n+1}}
$$
where 
\begin{equation}\label{F_K}
    F_R(z):= \exp\bigg(\sum_{k\leqslant R} \frac{X(k)}{\sqrt{k}}z^k\bigg).
\end{equation}
We start our proof by stating some tools.
\begin{lemma}{(Borel--Cantelli's First Lemma).} \label{borelcantelli} Let $(A_n)_{n \geqslant 1}$ be sequence of events. Assuming that $\sum_{n=1}^{+ \infty}\mathbb{P}[A_n] < + \infty $ then $\mathbb{P}[\limsup_{n \to + \infty} A_n]=0$.

\end{lemma}
\begin{proof}[Proof]
See theorem 18.1 in \cite{Gut}.
\end{proof}

\begin{lemma}\label{updated_hoeffding}
    Let $Z=(Z_n)_{1 \leqslant   n \leqslant   N}$ be a complex martingale difference sequence with respect to a filtration $\mathcal{F}=(\mathcal{F}_n)_{1 \leqslant  n \leqslant  N }$.
    We assume that for each $n$, $Z_n$ is bounded almost surely.
    Furthermore, assume that we have $|Z_n| \leqslant   S_n$ almost surely, where $(S_n)_{1 \leqslant  n \leqslant  N }$ is a real predictable process with respect to the same filtration. We set the event $\Sigma:= \bigg\{ \sum_{1\leqslant   n \leqslant   N}S_n^2 \leqslant   T\bigg\}$ where $T$ is a  deterministic constant.
    Then, for any $\varepsilon >0$,
    $$ \mathbb{P}\bigg[\bigg\{\bigg|\sum_{1\leqslant   n \leqslant   N}Z_n\bigg|\geqslant \varepsilon\bigg\}\, \bigcap \, \Sigma \bigg] \leqslant   2\exp\bigg(\frac{-\varepsilon^2}{10T}\bigg).$$
\end{lemma}
\begin{proof}[Proof]
See lemma 3.13 in \cite{Rachid}.
\end{proof}

\begin{lemma}\label{multiplicative_chaos}
    Let $R \geqslant 1$ be a real number and $F_R(z)$ as in \eqref{F_K}. Uniformly for $1/2 \leqslant q \leqslant 1$ and $1\leqslant r \leqslant {\rm e}^{1/R}$, we have
    $$
    \mathbb{E}\bigg[\bigg( \frac{1}{2\pi} \int_{0}^{2\pi} |F_R(r{\rm e}^{i\theta})|^2{\rm d}\theta \bigg)^q\bigg]\ll \bigg(\frac{R}{1+(1-q)\sqrt{\log R}}\bigg)^q.
    $$
\end{lemma}
\begin{proof}[Proof]
    See proposition 3.2 in \cite{model_zaman}.
\end{proof}
\section{Reduction of the problem}\label{section_reduction_chapter_2}
The goal of this section is to reduce the problem to something simple to deal with. We want to prove that the event
$$ \mathcal{A}:=\big\{ |A(n)|> 4(\log n )^{3/4 + \varepsilon}     \text{, for infinitely many }n\big\}$$
holds with null probability. We adopt the reasoning from \cite{Rachid} and set $X_{\ell}$ to be equal to $2^{\ell^K}$, where $K=\frac{25}{\varepsilon}$. It suffices to prove that the event 
$$ \mathcal{B}:=\bigg\{ \sup_{X_{\ell  -1}<n\leqslant X_{\ell}}\frac{|A(n)|}{(\log n )^{3/4 + \varepsilon}}> 4    \text{, for infinitely many }\ell\bigg\}$$ 
holds with null probability. We set
$$
\mathcal{B}_{\ell} := \bigg\{ \sup_{X_{\ell  -1}<n\leqslant X_{\ell}}\frac{|A(n)|}{(\log n )^{3/4 + \varepsilon}}> 4  \bigg\}.
$$
In order to prove Theorem \ref{theoreme_principal} using Borel-Cantelli's First Lemma~\ref{borelcantelli}, it is enough to establish the convergence of the series $ \sum_{\ell \geqslant 1} \mathbb{P}\big[\mathcal{B}_{\ell}\big]$.

Arguing as Lau--Tenenbaum--Wu in \cite{Tenenbaum} in the proof of lemma 3.1 and recently in \cite{Rachid} at the beginning of Section 5.1, we consider for every $j\geqslant 0 $
$$
y_j=\bigg\lfloor\frac{2^{\ell^K}{\rm e}^{j/\ell}}{2^{K\ell^{K-1}}}\bigg\rfloor \text{ and } \widetilde{y}_j:= \frac{2^{\ell^K}{\rm e}^{j/\ell}}{2^{K\ell^{K-1}}}.
$$
Let $J$ be minimal under the constraint $y_J \geqslant   X_{\ell}$ which means 
\begin{equation}\label{definition_J}
    J_{\ell}=J:=\lceil K \ell^{K} \log 2 \rceil \ll ~\ell^{K}.
\end{equation}
Note that  $\ell^{K} =\frac{1}{\log 2} \log X_{\ell} \asymp  \log n \asymp \log y_j $ for any $n \in ]X_{\ell -1},X_{\ell}]$ and~$~1\leqslant~  j\leqslant ~J$. \\
Let $n \in ]X_{\ell -1},X_{\ell}]$, we start by splitting $A(n)$ according to the size of $\lambda_1$ and $m_{\lambda_1}(\lambda)$, we divide $A(n)$ to 
$$
A(n) = A_0(n)+ A_1(n) +A_2(n)+ A_3(n)
$$
where
$$
 A_0(n):= \sum_{\substack{|\lambda|=n \\ \lambda_1 \leqslant y_0}} a(\lambda), \qquad\qquad 
A_1(n) :=\sum_{\substack{|\lambda|=n \\ \lambda_1 > y_0 \\ m_{\lambda_1}(\lambda)=1}} a(\lambda),
$$
$$
A_2(n):= \sum_{\substack{|\lambda|=n \\ \lambda_1 > y_0\\ m_{\lambda_1}(\lambda)= 2}} a(\lambda)
\qquad \text{and}\qquad 
A_3(n):= \sum_{\substack{|\lambda|=n \\ \lambda_1 > y_0\\ m_{\lambda_1}(\lambda)\geqslant 3}} a(\lambda).
$$
Now we set, for each $r\in \{0,1,2,3\}$
$$
\mathcal{B}_{\ell}^{(r)}:=  \bigg\{ \sup_{X_{\ell  -1}<n\leqslant X_{\ell}}\frac{|A_r(n)|}{(\log n )^{3/4 + \varepsilon}}> 1  \bigg\}.
$$
Note that $$\mathcal{B}_{\ell} \subset \bigcup_{r=0}^3\mathcal{B}_{\ell}^{(r)},$$ thus, to prove $ \sum_{\ell \geqslant 1} \mathbb{P}\big[\mathcal{B}_{\ell}\big]$ converges, it suffices to prove $ \sum_{\ell \geqslant 1} \mathbb{P}\big[\mathcal{B}_{\ell}^{(r)}\big]$ converges, for all $r\in\{0,1,2,3\}$.
\section{Convergence of \texorpdfstring{$\sum_{\ell \geqslant 1} \mathbb{P}[\mathcal{B}_{\ell}^{(0)}] \text{ and }$} \texorpdfstring{$\sum_{\ell \geqslant 1} \mathbb{P}[\mathcal{B}_{\ell}^{(3)}]$}.}
\noindent Let's start first by dealing with  $\mathcal{B}_{\ell}^{(0)}$.
\begin{lemma}\label{lemma_pb0}
The sum $ \sum_{\ell \geqslant 1} \mathbb{P}[\mathcal{B}_{\ell}^{(0)}]$ converges.
\end{lemma}
\begin{proof}[Proof]
We have 
\begin{equation}\label{equation A_0}
   \begin{aligned}
\mathbb{E}\big[ |A_0(n)|^2\big]  = \sum_{\substack{ |\lambda|=n \\ \lambda_1 \leqslant y_0}} \mathbb{E}\big[|a(\lambda)|^2\big]
 =  \sum_{\substack{ |\lambda|=n \\ \lambda_1 \leqslant y_0}}\prod_k \frac{1}{k^{m_k}m_k!}.
\end{aligned} 
\end{equation}
Arguing as Soundararajan and Zaman in \cite{model_zaman}, the right side of \eqref{equation A_0} is the coefficient of $z^n$ in the generating function $\exp\big(\sum_{k\leqslant y_0}z^k/k \big)$. Since the coefficients of this generating function are all non-negative, for any $r > 0$ we conclude that 
\begin{equation}\label{equaltion_with_r}
    \mathbb{E}\big[ |A_0(n)|^2\big]  = \sum_{\substack{ |\lambda|=n \\ \lambda_0 \leqslant y_0}} \prod_k \frac{1}{k^{m_k}m_k!} \leqslant \frac{1}{r^n}\exp\bigg(\sum_{k\leqslant y_0}\frac{r^{k}}{k} \bigg).
\end{equation}
By choosing $r=\exp(1/y_0)$ and since $X_{\ell-1}<n\leqslant X_{\ell}$, we get
\begin{align*}
    \mathbb{E}\big[ |A_0(n)|^2\big]  & \leqslant \frac{\exp\bigg(\sum_{k \leqslant y_0} \frac{{\rm e}^{k/y_0}}{k}\bigg)}{{\rm e}^{n/y_0}}
    \\ & \leqslant \frac{\exp\bigg(\sum_{k \leqslant y_0} \frac{{\rm e}}{k}\bigg)}{\exp(X_{\ell-1}/y_0)}
    \\ & \leqslant \frac{\exp\big(2{\rm e} \log y_0\big)}{\exp(X_{\ell-1}/y_0)}
    \\ & \leqslant \frac{y_0^{2{\rm e}}}{\exp(X_{\ell-1}/y_0)} \leqslant \frac{2^{2{\rm e}\ell^K}}{\exp\big(2^{c\, \ell^{K-2}}\big)}
\end{align*}
where $c$ is an absolute constant. Thus, by Markov's inequality
$$
\mathbb{P}[\mathcal{B}_{\ell}^{(0)}]\leqslant \sum_{X_{\ell-1}<n\leqslant X_{\ell}} \frac{1}{(\log n)^{1/2+2\varepsilon}} \mathbb{E}\big[ |A_0(n)|^2\big] \leqslant \frac{2^{(2{\rm e}+1)\ell^K}}{\exp\big(2^{c\,\ell^{K-2}}\big)}.
$$
It follows that the sum $ \sum_{\ell \geqslant 1} \mathbb{P}[\mathcal{B}_{\ell}^{(0)}]$ converges.
\end{proof}
\begin{lemma}
    The sum $ \sum_{\ell \geqslant 1} \mathbb{P}[\mathcal{B}_{\ell}^{(3)}]$ converges.
\end{lemma}
\begin{proof}[Proof]
We have
\begin{align*}
     \mathbb{E}\big[ |A_3(n)|^2\big] & = \sum_{\substack{|\lambda|=n \\ m_{\lambda_1}(\lambda) \geqslant 3 \\ \lambda_1 >y_0} } \mathbb{E}\big[ |a(\lambda)|^2\big]
     \\ & \leqslant \sum_{ y_0 < k\leqslant n/3} \frac{1}{k^3} \sum_{\substack{|\lambda|=n-3k \\ \lambda_1 \leqslant k} }\mathbb{E}\big[ |a(\lambda)|^2\big].
\end{align*}
Since, for every $k,n\geqslant 1$
$$
\sum_{\substack{|\lambda|=n-3k \\ \lambda_1 \leqslant k} }\mathbb{E}\big[ |a(\lambda)|^2\big] \leqslant \sum_{\substack{|\lambda|=n-3k } }\mathbb{E}\big[ |a(\lambda)|^2\big] \leqslant 1 ,
$$
we get
$$
\mathbb{E}\big[ |A_3(n)|^2\big] \leqslant \sum_{ y_0 < k\leqslant n/3} \frac{1}{k^3} \ll \frac{1}{y_0^2}.
$$
Thus,
$$
\mathbb{P}[\mathcal{B}_{\ell}^{(3)}] \leqslant \sum_{X_{\ell-1} <n \leqslant X_{\ell}} \frac{1}{(\log n)^{1/2+2\varepsilon}}\mathbb{E}\big[ |A_3(n)|^2\big] \ll 2^{\ell^K} \frac{2^{2K\ell^{K-1}}}{2^{2\ell^K}} = \frac{2^{2K\ell^{K-1}}}{2^{\ell^K}}.
$$
It follows that the sum $ \sum_{\ell \geqslant 1} \mathbb{P}[\mathcal{B}_{\ell}^{(3)}]$ converges.
\end{proof}
\section{Upper bound of \texorpdfstring{$\mathbb{P}[\mathcal{B}^{(1)}_{\ell}]$}.}\label{section_B_1}
In this subsection, we give a bound of  $ \mathbb{P}[\mathcal{B}_{\ell}^{(1)}]$. We consider the filtration $\big\{ \mathcal{F}_k \big\}_{k\geqslant 1}$, where $\mathcal{F}_k$ is the $\sigma$-algebra generated by $\{X(1),X(2),...,X(k-1)\}$. We set the convention $a(\lambda)=0$ for every $|\lambda|<0$. We have 
$$
A_1(n)= \sum_{\substack{|\lambda|=n \\ \lambda_1 > y_0 \\ m_{\lambda_1}(\lambda)=1}} a(\lambda) = \sum_{ y_0 <k \leqslant n} \frac{X(k)}{\sqrt{k}} \sum_{\substack{|\lambda|=n-k \\ \lambda_1 <  k }} a(\lambda).
$$
Note that $\sum_{\substack{|\lambda|=n-k \\ \lambda_1 <  k }} a(\lambda)$ is independent from $X(k)$ and depend only of $X(i)$ with $i<k$. Note, as well 
$$
\mathbb{E}\bigg[ \frac{X(k)}{\sqrt{k}} \sum_{\substack{|\lambda|=n-k \\ \lambda_1 <  k }} a(\lambda)\, \bigg| \mathcal{F}_{k} \,\bigg]=\mathbb{E}\bigg[ \frac{X(k)}{\sqrt{k}} \, \bigg| \mathcal{F}_{k} \,\bigg] \sum_{\substack{|\lambda|=n-k \\ \lambda_1 <  k }} a(\lambda) =0.
$$
Thus $A_1(n)$, as defined in section \ref{section_reduction_chapter_2}, is a sum of martingale differences. Recall that $y_j\asymp y_{j-1}$, we set
\begin{equation}\label{V_n}
    V(n) := \sum_{ y_0 <k \leqslant n} \frac{1}{k}\bigg| \sum_{\substack{|\lambda|=n-k \\ \lambda_1 <  k }} a(\lambda) \bigg|^2.
\end{equation}
We define
\begin{equation}\label{V_tilde_n}
    \widetilde{V}(n):= \sum_{\substack{ 1 \leqslant j\leqslant J \\ \frac{n}{y_{j}} >\ell^{100K}}}\sum_{ y_{j-1} <k \leqslant y_j}  \frac{1}{k}\bigg| \sum_{\substack{|\lambda|=n-k \\ \lambda_1 <  k }} a(\lambda) \bigg|^2
\end{equation}
and we set
\begin{equation}\label{V_n_y}
    V(n,y_{j}):= \frac{1}{y_{j}}\sum_{ y_{j-1} <k \leqslant y_j} \bigg| \sum_{\substack{|\lambda|=n-k \\ \lambda_1 <  k }} a(\lambda) \bigg|^2.
\end{equation}
Note that the number of $j$ such that $n\geqslant y_j$ and $\frac{n}{y_{j}} \leqslant \ell^{100K}$ is less than $100K \ell \log \ell+1$. We have then 
$$
\begin{aligned}\label{majoration_de_V}
    V(n) & \leqslant \widetilde{V}(n) + (100K \ell \log \ell+1) \sup_{\substack{ \frac{n}{y_{j}} \leqslant \ell^{100K} \\ n \geqslant y_j}}  V(n,y_{j}) 
    \\ & \leqslant C_0\bigg(  \widetilde{V}(n)+ \ell \log \ell \sup_{\substack{ 1\leqslant j\leqslant J}}  V(n,y_{j}) \bigg)
\end{aligned}
$$
where $C_0$ is a constant depend only on $K$. Let 
\begin{equation}\label{T(ell)_chapter2}
    T(\ell) =\ell^{10} \,\,\,\,\,\text{ and }\,\,\,\,\, T_1(\ell):= \frac{T(\ell)}{\ell \log\ell}.
\end{equation}
We set the events
\begin{equation}\label{event_T}
    \mathcal{T}(\ell)=:\mathcal{T} = \bigg\{ \sup_{ X_{\ell-1}<n \leqslant X_{\ell}}  V(n) \leqslant 2C_0T(\ell)\ell^{K/2} \bigg\}
\end{equation}
and
\begin{equation}\label{event_T_n}
    \mathcal{T}_n=\mathcal{T}_n(\ell) := \bigg\{   V(n) \leqslant 2C_0T(\ell)\ell^{K/2} \bigg\}.
\end{equation}
We define finally the following probabilities
\begin{equation}\label{probability_star}
    \mathbb{P}^{(1)}_{\ell} := \mathbb{P}\bigg[ \sup_{\substack{ X_{\ell-1} < n\leqslant X_{\ell} \\1 \leqslant j\leqslant J }} V(n,y_j)> T_1(\ell)\ell^{K/2}  \bigg]
\end{equation}
and
\begin{equation}\label{probability_tilde}
    \widetilde{\mathbb{P}}^{(1)}_{\ell} := \mathbb{P}\bigg[ \sup_{\substack{ X_{\ell-1} < n\leqslant X_{\ell} }} \widetilde{V}(n) > T(\ell)\ell^{K/2}  \bigg].
\end{equation}
It is  clear that $ \mathbb{P}[\overline{\mathcal{T}}] \leqslant \mathbb{P}^{(1)}_{\ell}  +\widetilde{\mathbb{P}}^{(1)}_{\ell}$, where $ \overline{\mathcal{T}}$ is the complement of $ \mathcal{T}$ in sample space.
\subsection{Bounding $\widetilde{\mathbb{P}}^{(1)}_{\ell} $.}
The objective of this section is to establish the convergence of the summation $\sum_{\ell \geqslant 1} \widetilde{\mathbb{P}}^{(1)}_{\ell}$.
\begin{lemma}\label{Ptilde_conv_tilde}
    The sum  $\sum_{\ell \geqslant 1} \widetilde{\mathbb{P}}^{(1)}_{\ell} $ converges.
\end{lemma}
\begin{proof}[Proof]
    Using the same argument as in the inequality \eqref{equaltion_with_r}, we have, for any $r>0$
    $$
    \mathbb{E}\bigg[ \bigg| \sum_{\substack{|\lambda|=n-k \\ \lambda_1 <  k }} a(\lambda) \bigg|^2 \bigg] \leqslant \frac{1}{r^{n-k}} \exp\bigg(\sum_{m <k} \frac{r^m}{m} \bigg).
    $$
    In particular for $r={\rm e}^{1/k}$, we have
    $$
    \mathbb{E}\bigg[ \bigg| \sum_{\substack{|\lambda|=n-k \\ \lambda_1 <  k }} a(\lambda) \bigg|^2 \bigg] \leqslant \frac{k^6}{\exp\big(\frac{n-k}{k}\big)} =\frac{{\rm e}k^6}{\exp\big(\frac{n}{k}\big)} .
    $$
    By using Markov's inequality and the observation that $T(\ell) \geqslant 1$, we can derive a bound by utilizing the inequality $y_j\leqslant X_{\ell}^2$, we get
    $$
    \begin{aligned}
        \widetilde{\mathbb{P}}^{(1)}_{\ell} &\leqslant \frac{1}{\ell^{k/2}}  \sum_{X_{\ell-1}<n\leqslant X_{\ell}}\sum_{\substack{ 1 \leqslant j\leqslant J \\ \frac{n}{y_{j}} >\ell^{100K}}}\sum_{ y_{j-1} <k \leqslant y_j}  \frac{1}{k}\mathbb{E}\bigg[ \bigg| \sum_{\substack{|\lambda|=n-k \\ \lambda_1 <  k }} a(\lambda) \bigg|^2 \bigg]
        \\ & \leqslant \frac{1}{\ell^{k/2}}  \sum_{X_{\ell-1}<n\leqslant X_{\ell}}\sum_{\substack{ 1 \leqslant j\leqslant J \\ \frac{n}{y_{j}} >\ell^{100K}}}\sum_{ y_{j-1} <k \leqslant y_j} \frac{{\rm e}k^5}{\exp\big(\frac{n}{k}\big)}
        \\ & \leqslant \frac{1}{\ell^{k/2}}  \sum_{X_{\ell-1}<n\leqslant X_{\ell}}\sum_{\substack{ 1 \leqslant j\leqslant J \\ \frac{n}{y_{j}} >\ell^{100K}}}\sum_{ y_{j-1} <k \leqslant y_j} \frac{{\rm e}2^{10\ell^{K}}}{{\rm e}^{\ell^{100K}}}
        \\ & \ll \frac{\ell^{K/2} 2^{13\ell^K}}{{\rm e}^{\ell^{100K}}}.
    \end{aligned}
    $$
    We deduce that the sum  $\sum_{\ell \geqslant 1} \widetilde{\mathbb{P}}^{(1)}_{\ell} $ converges.
\end{proof}
\subsection{Bounding \texorpdfstring{$\mathbb{P}^{(1)}_{\ell}$}.}\label{section_P_1}
The goal of this subsection is to give a bound of  $\mathbb{P}^{(1)}_{\ell}$. To do so, we follow exactly the same steps as Section \ref{low_moment_estimates} in Chapter 1.  We have for all $1\leqslant j\leqslant J$
$$
V(n,y_j) \leqslant U_j:=\frac{1}{{y}_j} \sum_{r=0}^{+\infty}\max_{y_{j-1}<\beta\leqslant y_j} \bigg| \sum_{\substack{|\lambda|=r \\ \lambda_1 \leqslant \beta}} a(\lambda)\bigg|^2
$$
and
$$
V(n,y_0) \leqslant U_j:=\frac{1}{{y}_j} \sum_{r=0}^{+\infty} \bigg| \sum_{\substack{|\lambda|=r \\ \lambda_1 \leqslant y_0}} a(\lambda)\bigg|^2
$$
Let 
$$
I_j:= \frac{1}{2\pi\widetilde{y}_j} \bigg(\frac{\widetilde{y}_{j}}{\widetilde{y}_{0}}\bigg)^{-1/\ell^K}\int_{0}^{2\pi } \big|F_{y_j}({\rm e}^{i\vartheta})\big|^2{\rm d}\vartheta.
$$
Define $\mathcal{S}$ to be the event $\big\{ I_{j} \leqslant   \frac{T_1(\ell)^{1/2}}{\ell ^{K/2}} \text{ for all } 0 \leqslant j\leqslant J \big\}$ and $\mathcal{S}_j:=\big\{ I_{j} \leqslant   \frac{T_1(\ell)^{1/2}}{\ell ^{K/2}}  \big\} $.
Now we have
$$
\begin{aligned}
    \mathbb{P}_{\ell}^{(1)} & \leqslant  \mathbb{P}\bigg[ \bigcup_{0\leqslant j\leqslant J}\bigg\{ 
U_j \geqslant  \frac{T(\ell) \ell^{K/2} }{  \ell \log \ell}   \bigg\} \bigcap \big\{ \mathcal{S}_{j-1}\big\}\bigg] + \mathbb{P}\big[\,\overline{\mathcal{S}}\,\big]
\\ & \leqslant  \sum_{j=0}^{J} \mathbb{P}\bigg[ \bigg\{ 
U_j \geqslant  \frac{T(\ell) \ell^{K/2} }{  \ell \log \ell}   \bigg\} \bigcap \big\{ \mathcal{S}_{j-1}\big\}\bigg] + \mathbb{P}\big[\,\overline{\mathcal{S}}\,\big].
\end{aligned}
$$
Let start by treating
$$
\widetilde{\mathbb{P}}_j:=\mathbb{P}\bigg[ \bigg\{ 
U_j \geqslant  \frac{T(\ell) \ell^{K/2} }{  \ell \log \ell}   \bigg\} \bigcap \big\{ \mathcal{S}_{j-1}\big\}\bigg].
$$
By Markov's inequality, we have
$$
\begin{aligned}
\widetilde{\mathbb{P}}_j & \leqslant \mathbb{P}\bigg[ \bigg\{ 
U_j \geqslant  \frac{T(\ell) \ell^{K/2} }{  \ell \log \ell}   \bigg\}\,  \bigg| \, \big\{ \mathcal{S}_{j-1}\big\}\bigg] 
\\ & \leqslant \frac{\ell \log \ell}{T(\ell) \ell^{K/2}} \mathbb{E}\big[ U_j \, \big|\,\mathcal{S}_{j-1}\big].
\end{aligned}
$$
We will need the following lemmas
\begin{lemma}\label{lemma_surmartingale_chapter_2}
    Let $r\geqslant 1$ be a fixed integer. The sequence $\bigg( \bigg| \sum_{\substack{|\lambda|=r \\ \lambda_1 \leqslant \beta}} a(\lambda)\bigg|\bigg)_{\beta}$ is a submartingale with respect to the filtration $\big(\mathcal{F}_{\beta}\big)_{\beta}$, where $\mathcal{F}_{\beta} $ is the $\sigma$-algebra generated by $\big\{ X(k); k\leqslant \beta \big\}$.
\end{lemma}
\begin{proof}[Proof]
    Let $r\geqslant 1$ and $\beta \geqslant 0$. By using the fact that $\int z^n {\rm d}\mathbb{P}=0$ for all $n\geqslant 1$ (which is equivalent to say $\mathbb{E}[(X(\beta+1))^n]=0$), we have 
    $$
    \begin{aligned}
        \mathbb{E}\bigg[\bigg| \sum_{\substack{|\lambda|=r \\ \lambda_1 \leqslant \beta+1}} a(\lambda)\bigg|  \, \big| \, \mathcal{F}_{\beta}\bigg] & = \int   \bigg|\sum_{n\geqslant 0} \bigg(\frac{z}{\sqrt{\beta+1}}\bigg)^n \sum_{\substack{|\lambda|=r-n(\beta+1) \\ \lambda_1 \leqslant \beta}} a(\lambda)\bigg|  {\rm d}\mathbb{P}
        \\ &  \geqslant \bigg| \int   \sum_{n\geqslant 0} \bigg(\frac{z}{\sqrt{\beta+1}}\bigg)^n \sum_{\substack{|\lambda|=r-n(\beta+1) \\ \lambda_1 \leqslant \beta}} a(\lambda)  {\rm d}\mathbb{P}\bigg|
        \\ & = \bigg| \sum_{\substack{|\lambda|=r \\ \lambda_1 \leqslant \beta}} a(\lambda)\bigg|.
    \end{aligned}
    $$
    This ends the proof.
\end{proof}
\noindent We have the analogue of Lemma \ref{P(S)} in Chapter 1.
\begin{lemma}\label{lemma_I_j_sousmartingale_chap2}
For $\ell$ large enough, the sequence $(I_j)_{j\geqslant 0}$ is supermartingale with respect to the filtration $(\mathcal{F}_{y_j})_{j\geqslant 0}$.
\end{lemma}
\begin{proof}[Proof]
    We have
    $$
    \begin{aligned}
        \mathbb{E}\big[I_{j} \, \big| \mathcal{F}_{y_{j-1}}\,\big] & = \frac{1}{2\pi\widetilde{y}_j} \bigg(\frac{\widetilde{y}_{j}}{\widetilde{y}_{0}}\bigg)^{-1/\ell^K}\int_{0}^{2\pi } \mathbb{E}\bigg[ \big|F_{y_j}({\rm e}^{i\vartheta})\big|^2\, \big| \, \mathcal{F}_{y_j}\bigg]{\rm d}\vartheta 
        \\ & = \frac{1}{2\pi\widetilde{y}_j} \bigg(\frac{\widetilde{y}_{j}}{\widetilde{y}_{0}}\bigg)^{-1/\ell^K}\int_{0}^{2\pi } \mathbb{E}\bigg[ \bigg| \exp\bigg( \sum_{y_{j-1}<k\leqslant y_j}\frac{X(k)}{\sqrt{k}} {\rm e}^{ik\vartheta}\bigg)\bigg|^2\bigg]\big|F_{y_{j-1}}({\rm e}^{i\vartheta})\big|^2{\rm d}\vartheta 
        \\ & = \frac{1}{2\pi\widetilde{y}_j} \bigg(\frac{\widetilde{y}_{j}}{\widetilde{y}_{0}}\bigg)^{-1/\ell^K} \exp\bigg( \sum_{y_{j-1}<k\leqslant y_j}\frac{1}{k} \bigg)\int_{0}^{2\pi }  \big|F_{y_{j-1}}({\rm e}^{i\vartheta})\big|^2{\rm d}\vartheta 
    \end{aligned}
    $$
   In the sake of readability, we set
   $$
   b_j:= {\rm e}^{-1/\ell} \bigg(\frac{\widetilde{y}_{j}}{\widetilde{y}_{j-1}}\bigg)^{-1/\ell^K} \exp\bigg( \sum_{y_{j-1}<k\leqslant y_j}\frac{1}{k} \bigg) 
   $$
    By collecting the previous computations together, we find
    $$
    \begin{aligned}
        \mathbb{E}\big[I_{j} \, \big| \mathcal{F}_{y_{j-1}}\,\big] \leqslant b_jI_{j-1}.
    \end{aligned}
    $$
    To end the proof it suffices to prove that $b_j\leqslant 1$. Recall that $\frac{\widetilde{y}_j}{\widetilde{y}_{j-1}}={\rm e}^{1/\ell}$.
    For $\ell$ large enough, we have
$$
\begin{aligned}
b_j & = {\rm e}^{-1/\ell} \bigg(\frac{\widetilde{y}_{j}}{\widetilde{y}_{j-1}}\bigg)^{-1/\ell^K} \exp\bigg( \sum_{y_{j-1}<k\leqslant y_j}\frac{1}{k} \bigg) 
\\ & = \exp \bigg( -\frac{1}{\ell} + \frac{-1}{\ell^{K+1}}+\sum_{y_{j-1} <   k \leqslant y_{j}} \frac{1}{k}  \bigg) 
 \\ & = \exp \bigg(\frac{-1}{\ell^{K+1}} +O\bigg(\frac{1}{y_0}\bigg) \bigg)
\end{aligned}
$$
Recall that $y_0=\frac{2^{\ell^K}}{2^{K\ell^{K-1}}}$. Thus, for large $\ell$, we have $b_j\leqslant 1$.
It follows then that $$\mathbb{E}\big[I_{j} \, \big| \,\mathcal{F}_{y_{j-1}}\big] \leqslant ~ I_{j-1}.$$
\end{proof}
Using Lemma \ref{lemma_surmartingale_chapter_2}, followed by Lemma \ref{doob2} for $r=2$
$$
\begin{aligned}
     \mathbb{E}\big[ U_j \, \big|\,\mathcal{S}_{j-1}\big] & =  \frac{1}{ y_j}\sum_{r=0}^{+\infty}\mathbb{E}\bigg[ \max_{y_{j-1}<\beta \leqslant y_j} \bigg| \sum_{\substack{|\lambda|=r \\ \lambda_1 \leqslant \beta}} a(\lambda)\bigg|^2  \, \big|\,\mathcal{S}_{j-1}\bigg] 
     \\& \leqslant  \frac{4}{ y_j}\sum_{r=0}^{+\infty}\mathbb{E}\bigg[  \bigg| \sum_{\substack{|\lambda|=r \\ \lambda_1 \leqslant y_j}} a(\lambda)\bigg|^2  \, \big|\,\mathcal{S}_{j-1}\bigg]
     \\ & \ll \mathbb{E}\big[I_j\,\big|\, \mathcal{S}_{j-1}\big].
\end{aligned}
$$
Now by using Lemma \ref{lemma_I_j_sousmartingale_chap2} we have $\mathbb{E}\big[I_j\,\big|\, \mathcal{F}_{j-1}\big]\leqslant I_{j-1}.$
Thus we have 
$$
\mathbb{E}\big[ U_j \, \big|\,\mathcal{S}_{j-1}\big] \ll \mathbb{E}\big[I_j\,\big|\, \mathcal{S}_{j-1}\big] = \mathbb{E}\bigg[\mathbb{E}\big[I_j\,\big|\, \mathcal{F}_{j-1}\big]\,\big|\, \mathcal{S}_{j-1}\bigg] \leqslant \mathbb{E}\big[I_{j-1}\,\big|\, \mathcal{S}_{j-1}\big]\leqslant \frac{T_1(\ell)^{1/2}}{\ell^{K/2}}.
$$
Thus, we get at the end
\begin{equation}\label{sum_P_j_tilde_chapter2}
    \sum_{j=1}^{J} \Tilde{\mathbb{P}}_j\ll \sum_{j=1}^{J} \frac{T_1(\ell)^{1/2}}{\ell^{K/2}} \frac{\ell \log \ell}{T(\ell) \ell^{K/2}} \ll \frac{T_1(\ell)^{1/2}\ell \log \ell}{T(\ell)}=\frac{\sqrt{\ell \log \ell}}{T(\ell)^{1/2}}.
\end{equation}
Recall that $T(\ell) = \ell^{10}$ (see \eqref{T(ell)_chapter2}), we have $\sum_{\ell \geqslant 1} \frac{\sqrt{\ell \log \ell}}{T(\ell)^{1/2}}$ converges.\\
Let's deal with $ \mathbb{P}\big[ \mathcal{S}\big]$.\\
We define $\mathcal{A}:= \big\{ I_0\leqslant \frac{T_1^{1/4}(\ell)}{\ell^{K/2}} \big\}$.
\begin{lemma}\label{P(S)_chapter2}
    For $\ell$ large enough, we have $\mathbb{P}\big[\,\overline{\mathcal{S}}\,\big]\ll\frac{1}{(T_1(\ell))^{1/6}}$.
\end{lemma}
\begin{proof}[Proof]
    Indeed, we have $$\mathbb{P}\big[\,\overline{\mathcal{S}}\,\big] \leqslant \mathbb{P}\bigg[\max_{0 \leqslant j\leqslant J} I_j > \frac{(T_1(\ell))^{1/2}}{\ell^{K/2}}\, \bigg|\, \mathcal{A} \bigg] +  \mathbb{P}\big[\,\overline{\mathcal{A}}\,\big].$$
Recall, by Lemma $\ref{lemma_I_j_sousmartingale_chap2}$, the sequence $(I_j)$ is supermartingale. Thus by lemma $\ref{doob}$, we have
$$
\mathbb{P}\bigg[\max_{0 \leqslant j\leqslant J} I_j > \frac{T_1^{1/2}(\ell)}{\ell^{K/2}}\, \bigg|\, \mathcal{A} \bigg] \leqslant \frac{\ell^{K/2}}{(T_{1}(\ell))^{1/2}} \mathbb{E}\big[I_0\, \big|\, \mathcal{A}\big] \leqslant \frac{1}{(T_{1}(\ell))^{1/4}}.
$$
\noindent The second part of the lemma follows easily from Markov's inequality and Lemma \ref{multiplicative_chaos} for $R=y_0$ and $q=2/3$. We have then
$$\mathbb{P}\bigg[I_{0} > \frac{T_1(\ell)^{1/4}}{\ell ^{K/2}} \bigg]\leqslant   \frac{\ell ^{\frac{K}{3}}\mathbb{E}[I_{0}^{\frac{2}{3}}]}{T_1(\ell)^{\frac{1}{6}}} \ll \frac{1}{T_1(\ell)^{1/6}}. $$
\end{proof}
\begin{prop}\label{converge_P_(1)_chapter2}
$ \sum_{\ell \geqslant 1 }\mathbb{P}^{(1)}_{\ell} $ converges.
\end{prop}
\begin{proof}[Proof]
By gathering Lemma \ref{P(S)_chapter2} and inequality \eqref{sum_P_j_tilde_chapter2}, we get 
$$
\mathbb{P}^{(1)}_{\ell} \ll \frac{1}{(T_1(\ell))^{1/6}}.
$$
Since $T_1(\ell) =\frac{T(\ell)}{\ell \log \ell} \gg \ell^8$.
Thus $ \sum_{\ell \geqslant 1 }\mathbb{P}^{(1)}_{\ell} $ converges.
\end{proof}
\subsection{Convergence of \texorpdfstring{$\sum_{\ell \geqslant 1} \mathbb{P}\big[ \mathcal{B}^{(1)}_\ell\big]$}.}\label{section_convergenceofPB1}
In this section, we follow the same steps as in \cite{Rachid} in section 6.8. 
\begin{prop}\label{convergence_B_(1)}
    The sum $\sum_{\ell \geqslant 1} \mathbb{P}\big[ \mathcal{B}^{(1)}_\ell\big]$ converges.
\end{prop}
\begin{proof}[Proof]
    We have
    $$
    \begin{aligned}
        \mathbb{P}\big[ \mathcal{B}^{(1)}_\ell\big] & \leqslant \mathbb{P}\bigg[\bigcup_{X_{\ell-1} < n \leqslant   X_{\ell}} \bigg\{  \frac{|A_1(n)|}{(\log n )^{3/4 + \varepsilon}}> 1  \bigg\}\bigcap \mathcal{T}_n \bigg]+ \mathbb{P}\big[\,\overline{\mathcal{T}} \,\big]
        \\ & \leqslant \sum_{X_{\ell-1} < n \leqslant   X_{\ell}} \mathbb{P}\bigg[ \bigg\{  \frac{|A_1(n)|}{(\log n )^{3/4 + \varepsilon}}> 1  \bigg\}\bigcap \mathcal{T}_n \bigg] + \mathbb{P}\big[\,\overline{\mathcal{T}} \,\big]
    \end{aligned}
    $$
    By Proposition \ref{converge_P_(1)_chapter2} and Lemma \ref{Ptilde_conv_tilde}, the sum $\sum_{\ell \geqslant 1}\mathbb{P}\big[\,\overline{\mathcal{T}(\ell)} \,\big]$ converges. By applying the  Lemma~\ref{updated_hoeffding}, and since by assumption $K\varepsilon = 25$, we have then
    $$
    \begin{aligned}
        \mathbb{P}\bigg[ \bigg\{  \frac{|A_r(n)|}{(\log n )^{3/4 + \varepsilon}}> 1  \bigg\}\bigcap \mathcal{T}_n \bigg] & \leqslant 2 \exp \bigg(\frac{-C_2 \ell^{3K/2+2\varepsilon K}}{\ell^{K/2}T(\ell)} \bigg)
        \\ &\leqslant  2 \exp \bigg(-C_2 \ell^{K+44}\bigg)
    \end{aligned}
    $$
    where $C_2>0$ is an absolute constant. Finally, by using Lemma ,  we get
      $$
      \mathbb{P}\big[\mathcal{B}_{\ell}^{(1)}\big] \ll  \exp \bigg( \log 2\ell^{K}-C_2\, \ell^{K} \ell^{44} \bigg) + \mathbb{P}\big[\,\overline{\mathcal{T}} \,\big]
      $$
      Thus the sum $\sum_{\ell \geqslant 1}\mathbb{P}\big[\mathcal{B}_{\ell}^{(1)}\big]$ converges.
\end{proof}

\section{Upper bound of \texorpdfstring{$\mathbb{P}[\mathcal{B}^{(2)}_{\ell}]$}.}
 In this subsection,  we give an upper bound of $\mathbb{P}[\mathcal{B}_{\ell}^{(2)}]$.
\subsection{Preliminaries.}
We start by some results.
$$
A_2(n)= \sum_{\substack{|\lambda|=n \\ \lambda_1 > y_0 \\ m_{\lambda_1}(\lambda)=2}} a(\lambda) = \sum_{ y_0 <k \leqslant n/2} \frac{X(k)^2}{k} \sum_{\substack{|\lambda|=n-2k \\ \lambda_1 <  k }} a(\lambda).
$$ 
Note, as in Section \ref{section_B_1}, $A_2(n)$ is a sum of martingale difference with respect to the same filtration $(\mathcal{F}_{k})_{k\geqslant 1}$. By following the same steps as in Section \ref{section_B_1}, the problem is reduced to study 
\begin{align*}
    W(n) &:= \sum_{ y_0 <k \leqslant n/2} \frac{1}{2k^2}\bigg| \sum_{\substack{|\lambda|=n-2k \\ \lambda_1 <  k }} a(\lambda) \bigg|^2
    \\ & \leqslant \frac{1}{2y_0} \sum_{ y_0 <k \leqslant n} \frac{1}{k}\bigg| \sum_{\substack{|\lambda|=n-k \\ \lambda_1 <  k/2 }} a(\lambda) \bigg|^2.
\end{align*}
We set
$$
V^{(2)}(n):= \sum_{ y_0 <k \leqslant n} \frac{1}{k}\bigg| \sum_{\substack{|\lambda|=n-k \\ \lambda_1 <  k/2 }} a(\lambda) \bigg|^2.
$$
One can see that $V^{(2)}(n)$ is similar to $V(n)$ introduced in \eqref{V_n} with a little difference over the sum. We define the analogues of $\widetilde{V}(n)$  and $V(n,y_j)$: 
$$
\widetilde{V}^{(2)}(n):= \sum_{\substack{ 1 \leqslant j\leqslant J \\ \frac{n}{y_{j}} >\ell^{100K}}}\sum_{ y_{j-1} <k \leqslant y_j}  \frac{1}{k}\bigg| \sum_{\substack{|\lambda|=n-k \\ \lambda_1 <  k/2 }} a(\lambda) \bigg|^2
$$
and
$$
V^{(2)}(n,y_{j}):= \frac{1}{y_{j}}\sum_{ y_{j-1} <k \leqslant y_j} \bigg| \sum_{\substack{|\lambda|=n-k \\ \lambda_1 <  k/2 }} a(\lambda) \bigg|^2.
$$
We have then as it was done in \eqref{majoration_de_V}
$$
\begin{aligned}
    V^{(2)}(n)  \leqslant C_0\bigg(  \widetilde{V}^{(2)}(n)+ \ell \log \ell \sup_{\substack{ 1\leqslant j\leqslant J}}  V^{(2)}(n,y_{j}) \bigg)
\end{aligned}
$$
where $C_0$ is an absolute constant. We set the events
\begin{equation}\label{event_T_2}
    \mathcal{T}^{(2)}= \mathcal{T}^{(2)}(e\ll):= \bigg\{ \sup_{ X_{\ell-1}<n \leqslant X_{\ell}}  V^{(2)}(n) \leqslant 2C_0T(\ell)\ell^{K/2} \bigg\}
\end{equation}
and
\begin{equation}\label{event_T_n_2}
    \mathcal{T}^{(2)}_n= \mathcal{T}^{(2)}_n(\ell) := \bigg\{  V^{(2)}(n) \leqslant 2C_0T(\ell)\ell^{K/2} \bigg\}.
\end{equation}
We define finally the analogue probabilities as in \eqref{probability_star} and \eqref{probability_tilde}
\begin{equation}\label{probability_star_2}
    \mathbb{P}^{(2)}_{\ell} := \mathbb{P}\bigg[ \sup_{\substack{ X_{\ell-1} < n\leqslant X_{\ell} \\1 \leqslant j\leqslant J }} V^{(2)}(n,y_j)> T_1(\ell)\ell^{K/2}  \bigg]
\end{equation}
and
\begin{equation}\label{probability_tilde_2}
    \widetilde{\mathbb{P}}^{(2)}_{\ell} := \mathbb{P}\bigg[ \sup_{\substack{ X_{\ell-1} < n\leqslant X_{\ell} }} \widetilde{V}^{(2)}(n) > T(\ell)\ell^{K/2}  \bigg].
\end{equation}
It is  clear that $ \mathbb{P}[\overline{\mathcal{T}}] \leqslant \mathbb{P}^{(2)}_{\ell}  +\widetilde{\mathbb{P}}^{(2)}_{\ell}$. 
\begin{lemma}\label{Ptilde_conv_2}
    The sum  $\sum_{\ell \geqslant 1} \widetilde{\mathbb{P}}^{(2)}_{\ell} $ converges.
\end{lemma}
\begin{proof}[Proof]
    Is the same proof as Lemma \ref{Ptilde_conv_tilde}.
\end{proof}
\noindent On the other hand, we have
$$
V^{(2)}(n,y_j) \ll U^{(2)}_j:=\frac{1}{y_j}  \sum_{r=0}^{+\infty} \max_{y_{j-1}<\beta \leqslant y_j}\bigg| \sum_{\substack{|\lambda|=r \\ \lambda_1 \leqslant \beta/2}} a(\lambda)\bigg|^2
$$
By following exactly the same steps of Section \ref{section_P_1} we get the analogue of Proposition \ref{converge_P_(1)_chapter2}.
\begin{lemma}\label{lemma_P_2_conv}
    For sufficiently large $\ell$, we have 
\begin{equation}
   \mathbb{P}_{\ell}^{(2)}  \ll \frac{1}{T_1(\ell)^{1/6}}.
\end{equation}
Furthermore, since $T_1(\ell)\gg \ell^8$, the sum $\sum_{\ell \geqslant 1} \mathbb{P}_{\ell}^{(2)} $ converges. 
\end{lemma}
\noindent We have, as well
\begin{lemma}\label{convergenceproba_T_2}
    The sum $\sum_{\ell \geqslant 1}\mathbb{P}\big[\,\overline{\mathcal{T}^{(2)}(\ell)} \,\big]$ converges.
\end{lemma}
\subsection{Convergence of \texorpdfstring{$\sum_{\ell \geqslant 1} \mathbb{P}\big[ \mathcal{B}^{(2)}_\ell\big]$}.}
This section is similar to \ref{section_convergenceofPB1}.
\begin{prop}
    The sum $\sum_{\ell \geqslant 1} \mathbb{P}\big[ \mathcal{B}^{(2)}_\ell\big]$ converges.
\end{prop}
\begin{proof}[Proof]
    We have
    $$
    \begin{aligned}
        \mathbb{P}\big[ \mathcal{B}^{(2)}_\ell\big] & \leqslant \mathbb{P}\bigg[\bigcup_{X_{\ell-1} < n \leqslant   X_{\ell}} \bigg\{  \frac{|A_2(n)|}{(\log n )^{3/4 + \varepsilon}}> 1  \bigg\}\bigcap \mathcal{T}_n^{(2)} \bigg]+ \mathbb{P}\big[\,\overline{\mathcal{T}^{(2)}} \,\big]
        \\ & \leqslant \sum_{X_{\ell-1} < n \leqslant   X_{\ell}} \mathbb{P}\bigg[ \bigg\{  \frac{|A_2(n)|}{(\log n )^{3/4 + \varepsilon}}> 1  \bigg\}\bigcap \mathcal{T}_n^{(2)} \bigg] + \mathbb{P}\big[\,\overline{\mathcal{T}^{(2)}} \,\big].
    \end{aligned}
    $$
    Recall that $T(\ell)=\ell^{10}$, which gives the convergence of the sum $\sum_{\ell \geqslant 1}\mathbb{P}\big[\,\overline{\mathcal{T}^{(2)}(\ell)} \,\big]$ by Lemma \ref{convergenceproba_T_2}. By applying the  Lemma~\ref{updated_hoeffding}, and since $W(n)\leqslant V^{(2)}(n)/y_0$, we have then
    $$
    \begin{aligned}
        \mathbb{P}\bigg[ \bigg\{  \frac{|A_2(n)|}{(\log n )^{3/4 + \varepsilon}}> 1  \bigg\}\bigcap \mathcal{T}_n^{(2)} \bigg] & \leqslant 2 \exp \bigg(\frac{-C_2 y_0\ell^{3K/2+2\varepsilon K}}{\ell^{K/2}T(\ell)} \bigg)
        \\ &\leqslant  2 \exp \bigg(-C_2 y_0 \ell^{K+44}\bigg)
    \end{aligned}
    $$
    where $C_2>0$ is an absolute constant. We have at the end
      $$
      \mathbb{P}\big[\mathcal{B}_{\ell}^{(2)}\big] \ll  \exp \bigg( \log 2\ell^{K}-C_2\,y_0 \ell^{K} \ell^{44} \bigg) + \mathbb{P}\big[\,\overline{\mathcal{T}^{(2)}} \,\big].
      $$
      Thus the sum $\sum_{\ell \geqslant 1}\mathbb{P}\big[\mathcal{B}_{\ell}^{(2)}\big]$ converges.
\end{proof}
%%%%%%%%%%%
\section*{Acknowledgement}
The author would like to thank his supervisor Régis de la Bretèche for his patient guidance, encouragement and the judicious advices he has provided throughout the work that led to this paper.

\bibliographystyle{abbrv}
\bibliography{references.bib}

\begin{center}
Université Paris Cité, Sorbonne Université
CNRS,\\
Institut de Mathématiques de Jussieu- Paris Rive Gauche,\\
F-75013 Paris, France\\
E-mail: \author{rachid.caich@imj-prg.fr}
\end{center}
\end{document}